\documentclass{amsart}
\usepackage[usenames,dvipsnames]{pstricks}
\usepackage[mathscr]{eucal}

\newcommand{\numberseries}{\bfseries}   

\newlength{\thmtopspace}                
\newlength{\thmbotspace}                
\newlength{\thmheadspace}               
\newlength{\thmindent}                  

\newtheoremstyle{fixed bf head,slanted body}
                {\thmtopspace}{\thmbotspace}{\slshape}
                {\thmindent}{\bfseries}{.}{\thmheadspace}
                {{\numberseries \thmnumber{#2\;}}\thmname{#1}\thmnote{ (#3)}}

\newtheoremstyle{variable bf head,slanted body}
                {\thmtopspace}{\thmbotspace}{\slshape}
                {\thmindent}{\bfseries}{.}{\thmheadspace}
                {{\numberseries \thmnumber{#2\;}}\thmname{#1}\thmnote{ #3}}

\newtheoremstyle{fixed bf head,upright body}
                {\thmtopspace}{\thmbotspace}{\upshape}
                {\thmindent}{\bfseries}{.}{\thmheadspace}
                {{\numberseries \thmnumber{#2\;}}\thmname{#1}\thmnote{ (#3)}}

\newtheoremstyle{numbered paragraph}
                {\thmtopspace}{\thmbotspace}{\upshape}
                {\thmindent}{\upshape}{}{\thmheadspace}
                {{\numberseries \thmnumber{#2.}}}

\usepackage[usenames,dvipsnames]{pstricks}
\usepackage[mathscr]{eucal}
\usepackage{amsfonts,amsmath,amssymb,amsthm,amscd,amsxtra}
\usepackage{enumerate,verbatim}
\usepackage[all,2cell,ps]{xy}
\usepackage[notcite, notref]{}
\usepackage[pagebackref]{hyperref}
\usepackage{todonotes}
\usepackage{mathptmx}

\DeclareMathOperator{\md}{\operatorname{\mathsf{mod}}}

\theoremstyle{plain} 

\newtheorem{thm}{Theorem}[section]

\newtheorem{prop}[thm]{Proposition}
\newtheorem{cor}[thm]{Corollary}

\newtheorem{lem}[thm]{Lemma}

\theoremstyle{definition}
\newtheorem{dfn}[thm]{Definition}

\newtheorem{obs}[thm]{Observation}

\newtheorem{eg}[thm]{Example}
\newtheorem{conj}[thm]{Conjecture}
\newtheorem{ques}[thm]{Question}

\newtheorem{rmk}[thm]{Remark}

\numberwithin{equation}{section}

\newcommand{\fm}{\mathfrak{m}}

\newcommand{\fp}{\mathfrak{p}}
\newcommand{\fq}{\mathfrak{q}}

\newcommand{\CC}{\mathbb{C}}

\newtheorem{chunk}[thm]{\hspace*{-1.065ex}\bf}

\DeclareMathOperator{\Ass}{Ass}

\def\G-dim{\operatorname{\mathsf{G-dim}}}

\def\pd{\operatorname{\mathsf{pd}}}
\def\depth{\operatorname{\mathsf{depth}}}
\def\edim{\operatorname{\mathsf{embdim}}}
\def\height{\operatorname{\mathsf{height}}}
\def\length{\operatorname{\mathsf{length}}}
\def\Hom{\operatorname{\mathsf{Hom}}}
\def\Tr{\mathsf{Tr}\hspace{0.01in}}
\def\Ext{\operatorname{\mathsf{Ext}}}
\def\Spec{\operatorname{\mathsf{Spec}}}
\def\Tor{\operatorname{\mathsf{Tor}}}

\def\dim{\operatorname{\mathsf{dim}}}

\def\m{\mathfrak{m}}
\def\p{\mathfrak{p}}

\def\e{\operatorname{\mathsf{e}}}

\DeclareMathOperator{\im}{im}

\DeclareMathOperator{\Soc}{Soc}
\DeclareMathOperator{\Supp}{Supp}

\def\E{\operatorname{E}}

\usepackage{datetime}

\def\urltilda{\kern -.15em\lower .7ex\hbox{\~{}}\kern .04em}
\def\urldot{\kern -.10em.\kern -.10em}\def\urlhttp{http\kern -.10em\lower -.1ex
\hbox{:}\kern -.12em\lower 0ex\hbox{/}\kern -.18em\lower 0ex\hbox{/}}

\begin{document}
\baselineskip=15pt


\title[On the ideal case of a conjecture of Huneke and Wiegand]{On the ideal case of a conjecture of \\ Huneke and Wiegand}

\author{Olgur Celikbas}
\address{Department of Mathematics \\
West Virginia University \\
Morgantown, WV 26506 U.S.A}
\email{olgur.celikbas@math.wvu.edu}

\author{Shiro Goto}
\address{Department of Mathematics, School of Science and Technology, Meiji University, 1-1-1 Higashi-mita, Tama-ku, Kawasaki 214-8571, Japan}
\email{shirogoto@gmail.com}

\author{Ryo Takahashi}
\address{Graduate School of Mathematics, Nagoya University, Furocho, Chikusaku, Nagoya,
464-8602, Japan}
\email{takahashi@math.nagoya-u.ac.jp}
\urladdr{http://www.math.nagoya-u.ac.jp/~takahashi/}

\author{Naoki Taniguchi}
\address{Global Education Center, Waseda University, 1-6-1 Nishi-Waseda, Shinjuku-ku, Tokyo 169-8050, Japan}
\email{naoki.taniguchi@aoni.waseda.jp}
\urladdr{http://www.aoni.waseda.jp/naoki.taniguchi/}


\setcounter{tocdepth}{1}

\begin{abstract} A conjecture of Huneke and Wiegand claims that, over one-dimensional commutative Noetherian local domains, the tensor product of a finitely generated, non-free, torsion-free module with its algebraic dual always has torsion. Building on a beautiful result of Corso, Huneke, Katz and Vasconcelos, we prove that the conjecture is affirmative for a large class of ideals over arbitrary one-dimensional local domains. Furthermore we study a higher dimensional analog of the conjecture for integrally closed ideals over Noetherian rings that are not necessarily local. We also consider a related question on the conjecture and give an affirmative answer for first syzygies of maximal Cohen-Macaulay modules.
\end{abstract}

\maketitle


\section{Introduction}

Throughout, unless otherwise stated, $R$ denotes a commutative Noetherian local ring with unique maximal ideal $\fm$ and residue field $k$. All $R$-modules are assumed to be finitely generated. For such an $R$-module $M$, we denote the algebraic dual $\Hom_R(M,R)$ of $M$ by $M^{\ast}$. 

Recall that an $R$-module $M$ is called \emph{torsion-free} if the natural map $M \to M\otimes_R Q$, where $Q$ is the total quotient ring of $R$, is injective. In general, if $M$ is \emph{torsionless}, i.e., if the natural map $M \to M^{\ast\ast}$ is injective, then $M$ must be torsion-free, but the converse may fail; see, for example, the paragraph following \cite[2.1]{CT2} for details.

In this paper we are concerned with the following long-standing conjecture of Huneke and Wiegand; see \cite[4.6 and the discussion following the proof of 5.2]{HW1}. 

\begin{conj} (Huneke and Wiegand \cite{HW1}) \label{conjHW} Let $R$ be a one-dimensional local ring and let $M$ be a finitely generated, torsion-free $R$-module. Assume $M$ has \emph{rank}, i.e., there is a nonnegative  integer $r$ such that $M_{\fp} \cong R_{\fp}^{\oplus r}$ for all associated prime ideals $\fp$ of $R$. If $M\otimes_{R}M^{\ast}$ is torsion-free, then $M$ is free.
\end{conj}

Huneke and Wiegand \cite[3.1]{HW1} established Conjecture \ref{conjHW} over hypersurface domains, but in general the conjecture is very much open, even for ideals over one-dimensional complete intersection domains of codimension at least two; however see, for example,  \cite{GTNL} for some recent promising work on the conjecture. 

We shall point out that Conjecture \ref{conjHW} holds if and only if the following holds; see \cite[8.6]{CW} and Theorem \ref{aus}.

\begin{conj}  \label{conj2HW} Let $R$ be a local ring satisfying Serre's condition $(S_2)$, and let $M$ be a finitely generated, torsion-free $R$-module which has rank. If $M\otimes_{R}M^{\ast}$ is reflexive, then $M$ is free.
\end{conj}

A result of Corso, Huneke, Katz and Vasconcelos \cite{CHKV} shows integrally closed $\fm$-primary ideals are Tor-rigid and $\pd$-test; see \ref{rt} and Theorem \ref{CHKV}. In particular Conjecture \ref{conjHW} is true for such ideals; see Corollary \ref{cor2}. On reading through the proof of their result, we discover that the integrally closed hypothesis can be replaced with a weaker assumption, namely the weakly $\fm$-full property; see \cite{CIST}. Although this observation merely follows from a slight modification of the argument given in \cite{CHKV}, the outcome pertaining to Conjecture \ref{conjHW} makes a significant difference: in general, it is quite difficult to determine whether a given ideal is integrally closed, even by using a computer software such as Macaulay2 \cite{GS}, but the weakly $\fm$-full property is easier to check. Furthermore examples of weakly $\fm$-full ideals are easy to construct. For example, if $J$ is an ideal of $R$, then $J:_R\fm$ is a weakly $\fm$-full ideal; see Remark \ref{RM}. In general, an integrally closed ideal is not necessarily weakly $\fm$-full; zero ideal in a field is such an example. However, if $R$ is a domain with infinite residue field, then each nonzero integrally closed ideal is weakly $\fm$-full; see \cite[2.4]{Goto1}. In Section 2 we record some of these observations and obtain:

\begin{prop} \label{propintro} Let $R$ be a one-dimensional local domain and let $I$ be a nonzero, proper ideal of $R$. Assume $I$ is weakly $\fm$-full. If $I\otimes_{R}I^{\ast}$ is torsion-free, then $I$ is principal and $R$ is a DVR.
\end{prop}

Recall that $\operatorname{Pic}R$ consists of the isomorphism classes of finitely generated projective $R$-modules $M$ such that $M_{\p} \cong R_{\p} $ for all $ \p \in \Spec(R)$; see, for example, \cite[11.3]{Ei}.  In Section 3, we prove the following as Theorem \ref{th1}, which is our first main result in this paper.

\begin{thm} \label{th1intro} Let $R$ be a Noetherian ring (not necessarily local) satisfying Serre's condition $(S_2)$, and let $I$ be an integrally closed ideal of $R$ of positive height. Then  $I$ is locally free of rank one (so represents a class in $\operatorname{Pic}R$) if and only if $I\otimes_RI^*$ is reflexive. Moreover, if either of the equivalent conditions holds, then a primary decomposition of $I$ is of the form $$I = \bigcap_{\p \in \Ass_R(R/I)}\p^{(n(\p))}$$ where, for each $\p \in \Ass_{R}(R/I)$, $n(\p)$ and $\fp^{(n(\p))}$ denote a positive integer and the symbolic power of $\fp$, respectively. 
\end{thm}

We have already mentioned that Conjecture \ref{conjHW}, and hence Conjecture \ref{conj2HW}, holds for integrally closed $\fm$-primary ideals due to \cite{CHKV}. Hence we should highlight that we do not assume the ring in Theorem \ref{th1intro} is local.

In \cite{GTNL}, it was shown that Conjecture \ref{conjHW} fails if $R$ is a Cohen-Macaulay local ring with canonical module $\omega$ and one replaces the tensor product 
$M\otimes_{R}M^{\ast}$ with $M \otimes_{R} M^{\dagger}$, where $M^{\dagger}=\Hom_R(M,\omega)$. However, even if $M^{\ast}$ is replaced with $M^{\dagger}$, the conjecture still remains open when the ring in question is Cohen-Macaulay with \emph{minimal multiplicity}. Motivated by this fact, in the last section, we discuss a question related to Conjecture \ref{conjHW} and prove the following as our second main result; see Question \ref{qn2} and Theorem \ref{1}.

\begin{thm}\label{1intro}
Let $R$ be a Cohen-Macaulay local ring with a canonical module $\omega$. Assume $R$ has minimal multiplicity. If $M$ is a first syzygy of a maximal Cohen-Macaulay $R$-module and $M\otimes_RM^{\dagger}\cong \omega$, then $M\cong R$.
\end{thm}

Let us note that, in Theorem \ref{1intro}, the isomorphism between $M\otimes_RM^{\dagger}$ and $\omega$ need not be the natural one. Prior to giving a proof of Theorem \ref{1intro}, we record an example showing that the isomorphism $M\otimes_RM^{\dagger}\cong \omega$ does not necessarily imply $M\cong R$, in general. Subsequently we give an application of Theorem \ref{1intro} that concerns reflexive ideals; see Corollary \ref{corend}.


\section{Tensoring with weakly $\fm$-full ideals}

In this section we give a proof of Proposition \ref{propintro} and discuss several homological properties of weakly $\fm$-full ideals. Recall that an ideal $I$ of $R$ is called \emph{$\fm$-full} if $\fm I:x=I$ for some $x\in \fm$. As mentiond by Watanabe in \cite{Junzo}, $\fm$-full ideals were first defined and studied by Rees (unpublished); see also \cite[2.1]{Goto1}. Motivated by this definition, a class of ideals is defined  as follows:

\begin{dfn} (\cite[3.7]{CIST}) Let $R$ be a local ring and let $I$ be an ideal of $R$. Then $I$ is said to be \emph{weakly $\fm$-full} ideal provided that $\fm I: \fm \subseteq I$, or equivalently, $\fm I: \fm =I$.
\end{dfn}

Notice it follows $I \subseteq \fm I: \fm  \subseteq \fm I:x$ so that each $\fm$-full ideal is weakly $\fm$-full. Examples of weakly $\fm$-full ideals include non-maximal prime ideals; see \cite{CIST}. Moreover we have:

\begin{rmk} \label{RM} Let $J$ be an ideal of $R$, and set $I=J:_R\fm$. Then, since $\fm I \subseteq J$, we have $\fm I :\fm \subseteq J: \fm=I$ so that $I$ is a weakly $\fm$-full ideal of $R$.
\end{rmk}

Prior to giving a proof of Proposition \ref{propintro}, we record examples of weakly $\fm$-full ideals that are not $\fm$-full. To obtain such examples, we need some preliminary results.

\begin{chunk} (Corso and Polini \cite[2.1 and 2.2]{CAPC}) \label{CP} Let $R$ be a Cohen-Macaulay local ring which is not regular. If $\fq$ is a parameter ideal of $R$ and $I=\fq : \fm$, then $I^2=\fq I$ and $\fm I=\fm \fq$.
\end{chunk}

\begin{chunk} (Goto \cite[2.2(2)]{Goto1}) \label{Gotomfull} Let $R$ be a local ring and let $J$ be an $\fm$-full ideal of $R$. If $I$ is an ideal of $R$ such that $J \subseteq I$ and $\length(I/J)$ is finite, then $\mu(I)\leq \mu(J)$, where $ \mu(M)$ denotes the cardinality of a minimal generating set of an $R$-module $M$.
\end{chunk}

The next lemma is straightforward so that we omit its proof.

\begin{lem} \label{new} Let $R$ be a local ring, and let $I$ and $J$ be ideals of $R$ such that $J \subseteq I$ and $\fm I= \fm J$. Then it follows $\mu(I)=\mu(J)+\length(I/J)$.
\end{lem}

\begin{prop} \label{thmfull} Let $R$ be a $d$-dimensional Cohen-Macaulay local ring and let $I=\fq : \fm$ for some parameter ideal $\fq$ of $R$. If $\edim R > d+r$, where $r$ is the Cohen-Macaulay type of $R$, i.e., $r=\length(\Ext_R^d(R/\fm,R))$, then $I$ is not $\fm$-full.
\end{prop}

\begin{proof} Note, as $\edim R \neq  d$, $R$ is not regular and so $I$ is a proper ideal of $R$. Moreover the Cohen-Macaulay type of $R$, which is the length of $I/\fq$, is $r$. Since Theorem \ref{CP} implies $\fm I = \fm\fq$, we conclude from Lemma \ref{new} that $\mu(I)=d+r$. Therefore, if $I$ is $\fm$-full, then it follows from Theorem \ref{Gotomfull} that $\edim R=\mu(\fm)\leq \mu(I)=d+r$, which contradicts our assumption. Thus $I$ is not $\fm$-full.
\end{proof}

We are now ready to give two examples of weakly $\fm$-full ideals that are not $\fm$-full:

\begin{eg} \label{eg1} Let $R=\CC[\![t^4,t^5,t^6]\!]$ and let $\fq=(t^4)$. Then $R$ is a one-dimensional complete intersection ring and $\fq$ is a parameter ideal of $R$. Set $I=\fq:\fm$. Then Proposition \ref{thmfull} shows that $I$ is weakly $\fm$-full but not $\fm$-full. In particular, $I$ is not integrally closed; see \cite[2.4]{Goto1}. Note, as $t^7\notin R$, it follows that $t^{11}\notin \fq$. On the other hand, since $t^{16}$ and $t^{12}$ belong to $\fq$, we see $t^{11}\in I$. Therefore $I=\fq+t^{11}R=(t^4,t^{11})$.
\end{eg}

\begin{eg} Let $R=\CC[\![t^7,t^9,t^{11}, t^{13}]\!]$ and let $q=(t^{14})$. Then $R$ is a one-dimensional local domain that is not Gorenstein, and $\fq$ is a parameter ideal of $R$. Set $I=\fq:\fm$. Then Proposition \ref{thmfull} shows that $I$ is weakly $\fm$-full but not $\fm$-full. In particular, $I$ is not integrally closed; see \cite[2.4]{Goto1}. Furthermore one can check that $I=(t^{14}, t^{29}, t^{31}, t^{33})$.
\end{eg}

We proceed by recalling the definition of \emph{rigid-test} and \emph{strongly-rigid} modules.

\begin{dfn}\label{rt} (\cite{CGSZ, LongLC}) Let $R$ be a local ring and let $M$ be a finitely generated $R$-module. 

(i) $M$ is called \emph{rigid-test} provided that $M$ is Tor-rigid and $\pd$-test, i.e., the following condition holds: if $N$ is a finitely generated $R$-module with $\Tor^R_n(M,N)=0$ for some $n\geq 1$, then $\Tor^R_i(M,N)=0$ for all $i\geq n$, and $\pd(N)<\infty$; see \cite[2.3]{CGSZ}.

(ii) $M$ is called \emph{strongly rigid} provided that the following condition holds: if $T$ is a finitely generated $R$-module with $\Tor^R_s(M,T)=0$ for some $s\geq 1$, then $\pd_R(T)<\infty$; see \cite[2.1]{LongLC}. $\qed$
\end{dfn}

Although each rigid-test module is strongly-rigid, it is unknown if  the converse holds; see \cite[2.5]{CGSZ} for details. A result of Huneke, Corso, Katz and Vascencelos \cite[3.3]{CHKV} states that integrally closed $\fm$-primary ideals are rigid-test, and hence strongly rigid. The aim of this section is to point out that their argument in fact shows the same result for weakly $\fm$-full ideals. This will allow us to obtain various examples supporting Conjecture \ref{conjHW}; see Corollary \ref{cor2}. Note that, if $k$ is infinite, then integrally closed ideals are $\fm$-full, and hence weakly $\fm$-full, but there are weakly $\fm$-full ideals that are not integrally closed; see \cite[2.4]{Goto1}, above (2.7 and 2.8), and Section 3.

We now proceed and slightly modify the proof given in \cite{CHKV} to prove that $\fm$-primary weakly $\fm$-full ideals are rigid-test; since this fact has already been established for integrally closed ideals, we include an argument -- for completeness -- only for weakly $\fm$-full ideals; cf.,  \cite[3.1 and 3.2]{CHKV}.

\begin{thm} (Corso, Huneke, Katz and Vasconcelos \cite[3.3]{CHKV}) \label{CHKV} Let $R$ be a local ring and let $I$ be an $\fm$-primary ideal of $R$. Assume $I$ is weakly $\fm$-full, or integrally closed. If $\Tor_t^R(M,R/I)=0$ for some $R$-module $M$ and nonnegative integer $t$, then $\pd(M)<t$. In particular $I$ is a rigid-test $R$-module.
\end{thm}

\begin{proof} If $t=0$, then $M=0$ and $\pd(M)=-\infty$. Thus we may assume $t$ is positive.
Consider a minimal free resolution of $M$:
$$F=\cdots \longrightarrow F_{t+1}  \longrightarrow F_t \stackrel{^{\partial_{t}}} \longrightarrow  F_{t-1} \to \cdots \to F_0 \to 0 $$
Tensoring $F$ with $R/I$, we obtain the complex:
$$\overline{F}=\cdots \longrightarrow \overline{F}_{t+1}  \longrightarrow \overline{F}_t \stackrel{^{\overline{\partial}_{t}}} \longrightarrow  \overline{F}_{t-1} \to \cdots \to \overline{F}_0 \to 0 $$
Assume $\overline{\partial}_{t}=0$. Then $(\im \partial_{t+1}+IF_t)/IF_t= \im \overline{ \partial}_{t+1} = \ker \overline{\partial}_{t}=F_t/IF_t$. Hence $\im \partial_{t+1}+IF_t=F_t$ so that $\im \partial_{t+1}=F_t$ by Nakayama's lemma, i.e., $\partial_{t}=0$ and $\partial_{t-1}$ is injective. This would show $\pd(M)<t$.  Hence we assume $\im \overline{\partial}_{t} \neq 0$ and seek a contradiction.

As $\im \partial_{t} \subseteq F_{t-1}$, we have $(\im \partial_{t}) I \subseteq I F_{t-1}$. Since $I$ is $\fm$-primary, there is a positive integer $s$ such that $(\im \partial_{t}) \fm^s\subseteq I F_{t-1}$. We may assume $s$ is the smallest integer with this property. So $(\im \partial_{t}) \fm^{s-1} \nsubseteq I F_{t-1}$. Let $u \in (\im \partial_{t}) \fm^{s-1}$ with $u \notin I F_{t-1}$. Then $\fm u \in I F_{t-1}$. Therefore, $\overline{u} \in \Soc(F_{t-1}/ I F_{t-1})=\{\alpha \in F_{t-1}/ I F_{t-1}: \alpha \fm =0\}$. Moreover $u \in (\im \partial_{t}) \fm^{s-1} \subseteq \im \partial_{t}$ so that $\overline{u} \in \im \overline{\partial}_{t}$. Consequently,
\begin{equation}
u \notin I F_{t-1} \text{ and } \overline{u} \in \Soc(F_{t-1}/ I F_{t-1}) \cap \im \overline{\partial}_{t} 
\end{equation}
Since $u\in (\im \partial_t)\fm^{s-1}\subseteq \im \partial_t$, we have $u=\partial_t(v)$ for some $v\in F_{t}$. 

Let $x\in \fm$. Then it follows from the fact $\overline{u} \in \Soc(F_{t-1}/ I F_{t-1})$ that $xu \in I F_{t-1}$. Since $xu=x\partial_{t}(v)$, we deduce that $x\partial_{t}(v) \in I F_{t-1}$. This gives the series of implications:
$$\partial_{t}(xv)\in I F_{t-1} \; \Longrightarrow \; \overline{\partial}_{t}(\overline{xv})=0 \; \Longrightarrow \; \overline{xv} \in \im(\overline{\partial}_{t+1}) \; \Longrightarrow \; \overline{\partial}_{t+1}(\overline{z})=\overline{ \partial_{t+1}(z) }=\overline{xv} \;\;  \text{ for some } \; \overline{z}\in F_{t+1}/IF_{t+1}.$$
Consequently, we have
$$\partial_{t+1}(z)-xv\in I F_t \; \Longrightarrow \; xv=\partial_{t+1}(z)+w_0 \;  \text{ for some } w_0\in IF_t \; \Longrightarrow \; x\partial_{t}(v)=\partial_{t}(w_0).$$
Since $xu=x\partial_{t}(v)$, we get $xu=\partial_{t}(w_0)$. Note that $\partial_{t}(w_0) \in \partial_{t}(IF_t)\subseteq I\fm F_{t-1}$. Therefore $xu\in  I\fm F_{t-1}$. As $x$ is arbitrary, we obtain $\fm u \subseteq \fm I F_{t-1}$. It follows that $u\in (\fm I: \fm)F_{t-1}=IF_{t-1}$.
\end{proof}

Next we discuss tensoring certain modules with strongly rigid ones. Recall that each rigid-test module is strongly rigid; see \ref{rt}.

\begin{obs} \label{obswmf} Let $R$ be a local ring with $\depth(R)\geq 1$, and let $M$ be a nonzero strongly rigid $R$-module such that $M_{\fp}$ is a free $R_{\fp}$-module for each associated prime ideal $\fp$ of $R$. If $N$ embeds into a free $R$-module and $M\otimes_{R}N$ is torsion-free, then $N$ is free.

To observe this, note there is a short exact sequence of the form $0 \to N  \to F \to C \to 0$, where $F$ is a free $R$-module. Applying $-\otimes_RM$, we obtain the injection $\Tor_1^R(C,M) \hookrightarrow M\otimes_{R}N$. Since $M_{\fp}$ is a free $R_{\fp}$-module for each associated prime ideal $\fp$ of $R$, we have that $\Tor_1^R(C,M)$ is torsion. As $M\otimes_{R}N$ is torsion-free, this implies $\Tor_1^R(C,M)$=0, i.e., $C$ is free. Hence we see that $N$ is free as the short exact sequence $0 \to N  \to F \to C \to 0$ splits.
\end{obs}

One can also use Observation \ref{obswmf} to determine that an ideal is \emph{not} weakly $\fm$-full.

\begin{eg} Let $R=\CC[\![t^4,t^5,t^6]\!]$, $I=(t^4,t^5)$ and let $J=(t^4, t^6)$. Then it follows that $\Tor_2^R(R/I,R/J)=0$, i.e., $I\otimes_{R}J$ is torsion-free; see, for example, \cite[4.3]{HW1}. Since $I$ and $J$ are not principal, we conclude from Theorem \ref{CHKV} and Observation \ref{obswmf} that $I$ and $J$ are not weakly $\fm$-full ideals. In fact $t^6\fm=(t^{10}, t^{11}, t^{12}) \subseteq \fm I =(t^8,t^9, t^{10}, t^{11})$ so that $t^6 \in \fm I:\fm$ and so $\fm I:\fm \nsubseteq I$. Similarly one can check that $J$ is not weakly $\fm$-full directly.
\end{eg}

\begin{lem} \label{l1} Let $R$ be a local ring such that $\depth(R)\leq 1$, and let $M$ be a finitely generated torsion-free $R$-module. Then $M^{\ast}$ is free if and only if $M$ is free.
\end{lem}

\begin{proof} Let $X$ be an indecomposable direct summand of $M$. It suffices to assume $X^{\ast}$ is free and prove $X$ is free. Note $X$ is torsion-free. Hence, if $X^{\ast}=0$, then  it follows from \cite[1.2.3(b)]{BH} that $X=0$. So we may assume $X^{\ast}\neq 0$. 

Consider a minimal presentation $F_1 \to F_0 \to X \to 0$ of $X$. This yields the following exact sequence
\begin{equation} \tag{\ref{l1}.1}
0 \to X^* \to F_0^* \to F_1^* \to \Tr X \to 0,
\end{equation}
where $\Tr X$ is the Auslander transpose of $X$. Since $X^{\ast}$ is free and $\depth(R)\leq 1$, (\ref{l1}.1) gives that $\pd(\Tr X)\leq 1$. Therefore $\Ext_R^2(\Tr X, R)=0$ and so the natural map $X \to X^{\ast\ast}$ is surjective; see \cite[2.8]{AuBr}. As $X^{\ast\ast}$ is free and $X$ is indecomposable, this map is an isomorphism so that $X$ is free.
\end{proof}

\begin{rmk} The conclusion of Lemma \ref{l1} has been recently established in \cite[3.9]{DEL} for the case where $R$ is Cohen-Macaulay and $\dim(R)\leq 1$; our argument extends \cite[3.9]{DEL} with a different and short proof; see also \cite[1.2]{Her} and \cite[Theorem 3]{MT} for some related results concerning Lemma \ref{l1}.
\end{rmk}

\begin{cor} \label{maincorollary} Let $R$ be a local ring such that $\depth(R)\leq 1$ and let $M$ be a torsion-free strongly rigid $R$-module. Assume $M_{\fp}$ is a free $R_{\fp}$-module for each associated prime ideal $\fp$ of $R$. If $M\otimes_RM^{\ast}$ is torsion-free, then $M$ is free.
\end{cor}

\begin{proof} We may assume $\depth(R)=1$: otherwise $M$ would be free. Notice $M^{\ast}$, being torsionless, embeds into a free module. Therefore, it follows from Observation \ref{obswmf} that $M^{\ast}$ is free. So $M$ is free by Lemma \ref{l1}.
\end{proof}

\begin{rmk} \label{rmk3} Let $R$ be a local ring of positive depth and let $I$ be an $\fm$-primary ideal of $R$. Assume $I$ is either integrally closed, or weakly $\fm$-full. Assume further $I$ is principal. Then, since $I$ contains a non zero-divisor, $I$ is a free $R$-module so that $\pd_R(R/I)\leq 1$. In particular, $\Tor_2^R(R/\fm, R/I)=0$. It now follows from Theorem \ref{CHKV} that $\pd_R(R/\fm)\leq 1$, i.e., $R$ is a DVR.
\end{rmk}

Our next result establishes Proposition \ref{propintro} advertised in the introduction.

\begin{cor} \label{cor2} Let $R$ be a local ring of depth one and let $I$ be an $\fm$-primary ideal of $R$ (e.g., $R$ is a one-dimensional local domain and $I$ is a nonzero proper ideal of $R$). Assume $I$ is either integrally closed, or weakly $\fm$-full. If $I\otimes_{R}I^{\ast}$ is torsion-free, then $R$ is a DVR and $I$ is principal.
\end{cor}

\begin{proof} This follows from Theorem \ref{CHKV}, Corollary \ref{maincorollary} and Remark \ref{rmk3}
\end{proof}

\begin{rmk} \label{rmk2} Let $R$ be a one-dimensional local domain that is not regular, and let $I=x:_R\fm$ for some $0\neq x \in \fm$. Consider $I' = \frac{I}{x} \subseteq {\rm Q}(R)$, where ${\rm Q}(R)$ is total ring of fractions of $R$. Then $I'$ is a fractional ideal of $R$. Notice we have that $I^2=xI$; see \ref{CP}. Therefore it follows 
\begin{equation}\tag{\ref{rmk2}.1}
I'=\frac{I}{x}= \bigcup_{n\geq 1}\frac{I^n}{x^n}=R\left[\frac{I}{x}\right]. 
\end{equation}
Here the second equality is due to the fact that $I^2=xI$, and the third equality follows from the definition of $\displaystyle{R\left[\frac{I}{x}\right]}$. 
The equality in (\ref{rmk2}.1) implies that $I'$ is a module finite $R$-algebra. So we have $R \subseteq I' \subseteq \overline{R}$, where $\overline{R}$ is the integral closure of $R$. It follows that $(I')^{\ast} \cong R:_RI' = \fm$, and hence $I^{\ast} \cong \fm$. Notice $I$ is a weakly $\fm$-full ideal of $R$ so that $I\otimes_{R}I^{\ast}$ cannot be torsion-free; see Remark \ref{RM} and Corollary \ref{cor2}. However, since $I^{\ast} \cong \fm$, one can conclude that the tensor product $I\otimes_{R}I^{\ast}$ has torsion without appealing to Corollary \ref{cor2}; see, for example, \cite[page 842]{CON}. \qed
\end{rmk}

In view of Remark \ref{rmk2}, we next construct an example of a weakly $\fm$-full ideal $L$, where $L^{\ast} \ncong \fm$. This, in particular,  indicates that the conclusion of Corollary \ref{cor2} is not trivial.

\begin{eg} \label{ex4} Let $R=k[\!|t^9, t^{11}, t^{13}, t^{14}, t^{15}, t^{17}]\!]$ and set $L=I:_R\fm$, where $I=(t^{26}, t^{30},t^{32})$. Then $L$ is weakly $\m$-full. One can see that $L=(t^{26}, t^{30}, t^{32}, t^{34}, t^{36}, t^{38}, t^{42})$. Let $L'$ be the $R$-module generated by $1,t^4, t^6, t^8, t^{10}, t^{12}, t^{16}$. Then $L' \cong L$ and $t^{26}L'=L$. Note also that, since $R \subseteq L'$, we have $L^{\ast}=\Hom_R(L,R) \cong R:_{{\rm Q}(R)}L'=R:_{R}L' $.

Suppose now $L^{\ast}\cong \fm$, and seek a contradiction. It follows that $\fm \cong R:_{R}L'$ and hence $\fm L'\subseteq R$. However $t^{19}=t^{13}t^{6} \in \fm L'$, but $t^{19} \notin R$. Therefore, $L^{\ast} \ncong \fm$. Furthermore, since $L$ is weakly $\fm$-full and $R$ is not a DVR, we conclude from Corollary \ref{cor2} that $L\otimes_{R}L^{\ast}$ is not torsion-free, i.e., $L\otimes_{R}L^{\ast}$ has torsion.
\end{eg}

There are many examples in the literature supporting Conjecture \ref{conjHW}. For example, the conjecture is known to be true for ideals over numerical semigroup rings that have multiplicity at most seven; see \cite[1.7]{GTNL}. Notice, in Example \ref{ex4}, the numerical semigroup ring has multiplicity nine. Therefore, to our best knowledge, $L$ is a new example of an ideal supporting Conjecture \ref{conjHW}. Furthermore, it is a nontrivial example in the sense that $L^{\ast} \ncong \fm$; see Remark \ref{rmk2}.

\section{On integrally closed ideals}

This section is devoted to a proof of Theorem \ref{th1intro}. For our argument we will make use of the next two results; the first one is due to Auslander and follows from the proof of \cite[3.3]{Au}.

\begin{chunk} (Auslander; see \cite[3.3]{Au} and also \cite[5.2]{HW1}) \label{aus} Let $R$ be a local ring satisfying Serre's condition $(S_2)$ and let $M\in \md R$ be a torsion-free $R$-module. Assume $M_{\p}$ is a free $R_{\p}$-module for each prime ideal $\p$ of $R$ with $\height(\p) \leq 1$. If $M \otimes_RM^*$ is reflexive, then $M$ is free.
\end{chunk}

\begin{lem} \label{th1lem} Let $R$ be a Noetherian ring (not necessarily local) satisfying Serre's condition $(S_1)$, and let $I$ be an ideal of $R$. Assume $IR_{\fp}\cong R_{\fp}$ for some $\p \in \Ass_R(R/I)$. Then $R_{\fp}$ is a one-dimensional Cohen-Macaulay ring. If, furthermore, $I$ is integrally closed, then $R_{\fp}$ is a $\mathrm{DVR}$.
\end{lem}

\begin{proof} Note that $\depth_{R_{\p}}(R_{\p}/IR_{\p})=0$. Since $IR_{\fp}$ is principal, we have $\depth(R_{\fp})=1$. As $R$ satisfies $(S_1)$, we conclude that $R_{\fp}$ is a one-dimensional Cohen-Macaulay ring. Since $I R_{\fp}$ is an integrally closed $\fp R_{\fp}$-primary ideal of $R_{\fp}$, we conclude from Remark \ref{rmk3} that $R_{\fp}$ is a DVR.
\end{proof}

Recall that the \emph{Picard group} $\operatorname{Pic}R$ of a Noetherian ring $R$ consists of the isomorphism classes $[M]$ of finitely generated projective $R$-modules $M$ such that $M_{\p} \cong R_{\p} $ for all $ \p \in \Spec(R)$; see, for example, \cite[11.3]{Ei}.

\begin{thm}\label{th1} Let $R$ be a Noetherian ring (not necessarily local) and let $I$ be an ideal of $R$ of positive height. Assume $R$ satisfies Serre's condition $(S_2)$. Then the following conditions are equivalent.
\begin{enumerate}[\rm(i)]
\item $I$ is integrally closed and $[I] \in \operatorname{Pic}R$. 
\item $R_\p$ is a $\mathrm{DVR}$ for every $\p \in \Ass_R(R/I)$ and $[I] \in \operatorname{Pic}R$.
\item $I$ is integrally closed and $I\otimes_RI^*$ is reflexive.
\end{enumerate}
Moreover, if one of the equivalent conditions holds, then a primary decomposition of $I$ is of the form $$I = \bigcap_{\p \in \Ass_R(R/I)}\p^{(n(\p))}$$ where $n(\p)\geq 1$ for every $\p \in \Ass_{R}(R/I)$ and $\fp^{(n(\p))}$ denotes a symbolic power of $\fp$.
\end{thm}

\begin{proof}
$(ii) \Rightarrow (iii)$: Let $\overline{I}$ denote the integral closure of $I$. Suppose that $\overline{I}/I \ne (0)$ and choose $\p \in \Ass_{R}(\overline{I}/I)$. Then $\fp \in \Ass_R(R/I)$ so that $R_{\p}$ is a $\mathrm{DVR}$ by assumption. Hence $IR_\p = \overline{IR_\p} = \overline{I}R_\p$. This is a contradiction since $(\overline{I}/I)_{\p} \ne (0)$. Thus $\overline {I} = I$, i.e., $I$ is integrally closed.
Note that, since $[I] \in \operatorname{Pic}R$, $I$ is projective. Since $R$-duals and tensor products of projective modules are projective, and projective modules are reflexive, we conclude that $I\otimes_RI^{\ast}$ is reflexive.

$(iii) \Rightarrow (ii)$: We start the proof by proving the following claim.\\
\textbf{Claim.} Let $\fp \in \Spec(R)$ with $\height(\fp) \leq 1$. Then $IR_{\fp} \cong R_{\fp}$ and $\Supp_R(I)=\Spec(R)$.\\
\textbf{Proof of the claim}: Assuming $IR_{\fp} \cong R_{\fp}$ for all $\fp \in \Spec(R)$ with $\height(\fp) \leq 1$, we have that $\Supp_R(I)=\Spec(R)$: this is because each prime ideal of $R$ contains a minimal prime, which supports the $R$-module $I$. So we will prove $IR_{\fp} \cong R_{\fp}$ for each $\fp \in \Spec(R)$ with $\height(\fp) \leq 1$. This is clear if $I \not\subseteq \p$. Hence we assume $I \subseteq \p$. 

Since $I \subseteq \p$ and $I$ has positive height, we have $\dim R_\p = 1$. Since $R$ satisfies $(S_2)$, we conclude that $R_\p$ is a one-dimensional Cohen-Macaulay local ring. Moreover $IR_{\fp}$ has positive height since $I$ has positive height. In particular $\height_{R_{\fp}}(IR_{\fp})=1$ so that $\sqrt{IR_\p} = \p R_\p$.

As $I$ is integrally closed, $IR_\p$ is an integrally closed ideal of $R_{\fp}$. Morever $IR_\p \otimes_{R_\p}(IR_\p)^{*}$ is torsionfree over $R_{\p}$. Thus we see from Corollary \ref{cor2} that $IR_\p$ is principal, i.e., $IR_{\fp} \cong R_\p$. This proves the claim. $\blacksquare$

Now we proceed to show $[I] \in \operatorname{Pic}R$ by using \ref{aus}. For that fix a prime ideal $\fq$ of $R$. Pick $P\in \Spec(R_{\fq})$ with $\dim(R_P)\leq 1$.
Then $P=\fp R_{\fq}$ for some $\fp \in \Spec(R)$ with $\dim(R_{\fp})=\height(\fp) \leq 1$. It follows from the claim that $(IR_\fq)_{P}\cong IR_{\fp}\cong R_{\fp}$. Moreover $IR_{\fq}\otimes_{R_{\fq}}(IR_{\fq})^{\ast}$ is reflexive. So \ref{aus} implies $IR_{\fq}$ is free over $R_{\fq}$. Since $\Supp_R(I)=\Spec(R)$ by the claim, we see $IR_{\fq}\cong R_{\fq}$. This shows $I$ is projective, i.e., $[I] \in \operatorname{Pic}R$. 

Now let $\p \in \Ass_R(R/I)$. Since $[I] \in \operatorname{Pic}R$, we have $IR_{\fp}\cong R_{\fp}$. Thus, by Lemma \ref{th1lem}, we see $R_{\fp}$ is a one-dimensional Cohen-Macaulay ring. As $IR_{\fp}$ is  principal, i.e., free, we have $\Tor_{i}^{R_{\fp}} (R_{\fp}/\fp R_{\fp}, R_{\fp}/ IR_{\fp})=0$ for all $i\geq 2$.
As $IR_{\fp}$ is integrally closed and $\fp R_{\fp}$-primary, we deduce from Theorem \ref{CHKV} that $R_\p$ is a $\mathrm{DVR}$. This completes the proof of $(iii) \Rightarrow (ii)$.

$(ii) \Rightarrow (i)$: Since (ii) implies (iii), we see that (ii) implies (i).  

$(i) \Rightarrow (ii)$: This implication follows from Lemma \ref{th1lem}.

For the last assertation on the primary decomposition of $I$, let $\p \in \Ass_R(R/I)$. Then, since $R_{\fp}$ is a $\mathrm{DVR}$, we have $IR_\p = \p^{n(\p)}R_{\p}$ for some $n(\p)\geq 1$. Thus $IR_\p \cap R = \p^{(n(\p))}$, and hence the result follows.
\end{proof}


\section{On a question related to Conjecture \ref{conjHW}} 

Goto, Takahashi, Taniguchi and Truong \cite{GTNL} considered Conjecture \ref{conjHW} for ideals $I$ by replacing the tensor product $I\otimes_{R}I^{\ast}$ with $I \otimes_{R} I^{\dagger}$, where $\omega$ is the canonical module and $(-)^{\dagger}=\Hom(-, \omega)$. They proved that, this new version of Conjecture \ref{conjHW}, fails over a one-dimensional numerical semigroup ring; see \cite[7.3]{GTNL}. However the ring considered in their example does not have minimal multiplicity; recall that a Cohen-Macaulay local ring $R$ is said to have {\em minimal multiplicity} if $\e(R)=\edim R-\dim R+1$, where $\e(R)$ and $\edim R$ stand for the multiplicity of $R$ (with respect to $\fm$) and the embedding dimension of $R$, respectively. Therefore the following question still remains open:

\begin{ques} \label{qn1} Let $R$ be a one-dimensional local domain with a canonical module $\omega$, and let $I$ be an ideal of $R$. Assume $R$ has minimal multiplicity. If $I \otimes_{R} I^{\dagger}$ is torsion-free, then must $I\cong R$  or $I\cong \omega$? 
\end{ques}

In this section we consider a version of Question \ref{qn1}, and look at maximal Cohen-Macaulay modules $M$ such that $M\otimes_RM^{\dagger} \cong \omega$. More precisely, we ask:

\begin{ques} \label{qn2} Let $R$ be a Cohen-Macaulay local ring with a canonical module $\omega$. Assume $R$ has minimal multiplicity. If $M$ is a maximal Cohen-Macaulay $R$-module such that $M\otimes_RM^{\dagger} \cong \omega$, then must $M \cong R$ or $M \cong \omega$?
\end{ques}

Note that, if $M\otimes_RM^{\dagger} \cong \omega$, then $M\otimes_RM^{\dagger}$ is torsion-free, but somewhat surprisingly this isomorphism does not necessarily force $M \cong R$ or $M\cong \omega$, in general. For example, if $R=k[\!|t^9, t^{10}, t^{11}, t^{12}, t^{15}]\!]$ and $I=R+Rt$, then $I\otimes_RI^{\dagger} \cong \omega$, where $\omega=R+Rt+Rt^3+Rt^4$; see \cite[2.5]{GT}. 

The following is our main result in this section; it gives a partial affirmative answer to Question \ref{qn2}.

\begin{thm}\label{1}
Let $R$ be a Cohen-Macaulay local ring with a canonical module $\omega$. Assume $R$ has minimal multiplicity. Let $M$ be a first syzygy of a maximal Cohen-Macaulay $R$-module.
If $M\otimes_RM^{\dagger}\cong \omega$, then $M\cong R$.
\end{thm}

\begin{proof} We may assume, by replacing $R$ by $R[t]_{\m R[t]}$ with an indeterminate $t$ over $R$, that the residue field $k$ of $R$ is infinite. Therefore, since $R$ has minimal multiplicity, there exists a parameter ideal $\fq$ of $R$ such that $\m^2=\fq\m$; see, \cite[Theorem 1]{Sally}.
There are isomorphisms:
\begin{align*}
\omega/\fq \omega &\cong(M\otimes_RM^{\dagger})\otimes_RR/\fq\\
&\cong M/\fq M\otimes_{R/\fq}(M^{\dagger}\otimes_RR/\fq)\\
&\cong M/\fq M\otimes_{R/\fq}\Hom_{R/\fq}(M/\fq M, \omega /\fq \omega),
\end{align*}
Here the last isomorphism follows from \cite[3.3.3(a)]{BH} since $M$ is maximal Cohen-Macaulay and $\Ext^i_R(M,\omega)=0$ for all $i\geq 1$.
As $\omega /\fq \omega$ is isomorphic to the injective hull $\E_{R/\fq }(k)$, we obtain:
$$
E_{R/\fq}(k)\cong M/\fq M\otimes_{R/\fq}\Hom_{R/\fq}(M/\fq M,\E_{R/\fq }(k)).
$$
Taking the Matlis dual of this isomorphism over $R/\fq$, we see:
\begin{align*}
R/\fq &\cong\Hom_{R/\fq}(M/\fq M\otimes_{R/\fq}\Hom_{R/\fq}(M/\fq M,\E_{R/\fq}(k)),\E_{R/\fq}(k))\\
&\cong\Hom_{R/\fq}(M/\fq M,\Hom_{R/ \fq}(\Hom_{R/\fq}(M/\fq M,\E_{R/\fq}(k)),\E_{R/\fq}(k)))\\
&\cong\Hom_{R/\fq}(M/\fq M,M/\fq M).
\end{align*}
Now, by the hypothesis, there is an exact sequence
$$
0 \to M \to R^{\oplus a} \to N \to 0
$$
of $R$-modules such that $N$ is maximal Cohen--Macaulay.
Tensoring this short exact sequence with $R/\fq$, we have an exact sequence:
$$
0 \to \Tor_1^R(R/\fq,N) \to M/\fq M \to (R/\fq)^{\oplus a}.
$$
Notice $\Tor_1^R(R/\fq,N)$ is isomorphic to the first Koszul homology of $N$ with respect to the minimal system of generators of $\fq$, which is a regular sequence on $N$.
Thus $\Tor_1^R(R/\fq,N)$ vanishes, and this yields the injection $M/\fq M \hookrightarrow (R/\fq)^{\oplus a}$.
Therefore there is an isomorphism
$$
M/\fq M\cong L\oplus(R/\fq)^{\oplus b},
$$
where $L$ is a submodule of $\m(R/\fq)^{\oplus(a-b)}$.
Since $\m^2=\fq \m$, the module $L$ is a $k$-vector space so that $L\cong k^{\oplus c}$ for some $c\ge0$.
Now we obtain the isomorphisms:
\begin{align*}
R/\fq&\cong\Hom_{R/\fq}(M/\fq M,M/\fq M)\\
&\cong\Hom_{R/\fq}(k^{\oplus c}\oplus(R/\fq)^{\oplus b},k^{\oplus c}\oplus(R/\fq)^{\oplus b})\\
&\cong k^{\oplus(c^2+bc+bcr)}\oplus(R/\fq)^{\oplus b^2},
\end{align*}
where $r$ denotes the type of $R$.

Note that $R/\fq$ and $k$ are indecomposable $R/\fq$-modules, and that decomposition of each $R/\fq$-module into indecomposable $R/\fq$-modules is unique up to isomorphisms.
Hence we have either $(b,c)=(0,1)$ or $(b,c)=(1,0)$.
In both cases we obtain an isomorphism $M/\fq M\cong R/\fq$.
(In the former case, we also have $\fq=\m$ so that $R$ is regular).
Now applying \cite[1.3.5]{BH} repeatedly, we conclude that $M\cong R$.
\end{proof}

We finish this section with two corollaries of Theorem \ref{1}.

\begin{cor}\label{2}
Let $R$ be a $d$-dimensional Cohen--Macaulay local ring with a canonical module $\omega$. Assume $R$ has minimal multiplicity. Let $M$ be a $(d+1)$st syzygy of a finitely generated $R$-module.
If $M\otimes_RM^{\dagger}\cong \omega$, then $M\cong R$.
\end{cor}

\begin{proof} The assertion follows from Proposition \ref{1} since a $d$-th syzygy of a finitely generated $R$-module is a maximal Cohen--Macaulay $R$-module.
\end{proof}

\begin{cor}\label{corend}
Let $R$ be a one-dimensional Cohen--Macaulay local ring with a canonical module $\omega$. Assume $R$ has minimal multiplicity. If $I$ is a reflexive ideal of $R$ and $I\otimes_RI^{\dagger} \cong \omega$, then $I \cong R$.
\end{cor}

\begin{proof} Let $Q \xrightarrow{f} P \to I^* \to 0$ be a presentation of the $R$-module $I^*$ by finitely generated free $R$-modules $P$ and $Q$, where $I^*=\Hom_R(I,R)$.
Since $I$ is reflexive, dualizing this presentation by $R$, we obtain the exact sequence $0 \to I \to P^* \xrightarrow{f^*} Q^*$. Hence $I$ is a second syzygy of the cokernel of $f^*$, and so Corollary \ref{2} completes the proof.
\end{proof}

\section*{Acknowledgments}
The authors are grateful to C\u{a}t\u{a}lin Ciuperc\u{a} for useful discussions on weakly $\fm$-full ideals. The authors are also grateful to the referee for his/her careful reading, and for many helpful comments, corrections and suggestions.

Part of this work was completed when Celikbas visited the Meiji University in May and June 2017, supported by the JSPS Grant-in-Aid for Scientific Research (C) 26400054. He is grateful for the kind hospitality of the Meiji Department of Mathematics, and for the generous support of the JSPS grant.

Goto was partially supported by the JSPS Grant-in-Aid for Scientific Research (C) 16K05112.

Takahashi was partially supported by the JSPS Grant-in-Aid for Scientific Research (C) 16K05098. 

Taniguchi was partially supported by the JSPS Grant-in-Aid for Young Scientists (B) 17K14176.

\bibliography{a}
\bibliographystyle{amsplain}

\end{document}